\documentclass[a4paper,10pt]{article}
\usepackage{amsfonts,amsmath}
\usepackage[utf8]{inputenc}
\usepackage[T1]{fontenc}
\usepackage{graphicx}
\usepackage{geometry}
\usepackage{float}
\usepackage{amsmath}
\usepackage{amsthm}
\usepackage{amsfonts}
\usepackage{amssymb}
\usepackage{MnSymbol}
\usepackage{mathdots}
\usepackage{stmaryrd}
\usepackage{enumitem}
\usepackage{dutchcal}
\usepackage{bm}
\usepackage[]{caption}

\newtheorem{prop}{Proposition}[section]
\newtheorem{lemma}{Lemma}[section]
\newtheorem{thm}{Theorem}[section]
\newtheorem{cor}{Corollary}[section]
\newtheorem{fact}{Fact}[section]
\newtheorem{rem}{Remark}[section]

\newtheorem{df}{Definition}[section]
\newtheorem{nt}{Notation}[section]

\newcommand{\tr}{\mathrm{Tr}}

\newcommand{\C}{\mathbb{C}}
\newcommand{\R}{\mathbb{R}}
\newcommand{\A}{\mathbb{A}}

\newcommand{\dist}{-\mathrm{dist}}
\renewcommand{\d}{\delta}

\renewcommand{\S}{\mathcal{S}}

\renewcommand{\P}{\mathcal{P}}
\newcommand{\Hom}{\mathrm{Hom}}

\newcommand{\Irr}{\mathrm{Irr}}

\newcommand{\Res}{\mathrm{Res}}
\newcommand{\Ker}{\mathrm{Ker}}
\newcommand{\Vect}{\mathrm{Vect}}

\newcommand{\K}{\mathrm{K}}
\renewcommand{\tr}{\mathrm{tr}}
\newcommand{\GL}{\mathrm{GL}}

\newcommand{\Ind}{\mathrm{Ind}}
\newcommand{\ind}{\mathrm{ind}}
\newcommand{\M}{\mathcal{M}}
\newcommand{\m}{\mathcal{m}}
\newcommand{\B}{\mathcal{B}}
\newcommand{\V}{\mathcal{V}}
\renewcommand{\H}{\mathcal{H}}
\newcommand{\Ri}{\mathcal{R}}
\newcommand{\N}{\mathcal{N}}
\newcommand{\W}{\mathcal{W}}
\newcommand{\w}{\mathcal{w}}
\renewcommand{\v}{\mathcal{v}}
\renewcommand{\u}{\mathcal{u}}
\newcommand{\x}{\mathcal{x}}
\newcommand{\y}{\mathcal{y}}
\newcommand{\z}{\mathcal{z}}
\renewcommand{\l}{\lambda}

\newcommand{\U}{\mathcal{U}}
\newcommand{\Z}{\mathcal{Z}}
\newcommand{\Q}{\mathbb{Q}}

\title{Symmetric periods for automorphic forms on unipotent groups}
\author{N. Matringe}

\begin{document}

\maketitle

\begin{abstract}
Let $k$ be a number field and $\A$ be its ring of adeles. Let $U$ be a unipotent group defined over $k$, and $\sigma$ a $k$-rational involution of $U$ with fixed points $U^+$. As a consequence of the results of Moore (\cite{Moore65}), the space $L^2(U(k)\backslash U_{\A})$ is multiplicity free as a representation of $U_{\A}$. Setting \[p^+:\phi\mapsto \int_{U^+(k)\backslash 
{U}_{\A}^+} \phi(u)du\] to be the period integral attached to $\sigma$ on the space of smooth vectors of $L^2(U(k)\backslash U_{\A})$, we prove that if $\Pi$ is a topologically irreducible subspace of 
$L^2(U(k)\backslash U_{\A})$, then $p^+$ is nonvanishing on the subspace of smooth vectors in $\Pi$ if and only if $\Pi^\vee=\Pi^\sigma$. This is a global analogue of the local results in \cite{B1}, \cite{B2} and \cite{Mat20}, on which the proof relies. 
\end{abstract}

\section*{Introduction}

Let $k$ denote a number field, $\A$ its ring of Adeles, $G$ an algebraic group defined over $k$ and $\sigma$ a $k$-rational involution of $G$ with fixed points $G^+$. Suppose that the automorphic quotient $G^+(k)\backslash G^+_\A$ carries a right $G^+_\A$-invariant measure. Studying the non-vanishing of (possible regularizations of) the linear form $p^+:\phi\mapsto \int_{G^+(k)\backslash G^+_\A} \phi(g)dg$ on topologically irreducible subspaces of $L^2(G(k)\backslash G_\A)$ is a very popular topic when $G$ is reductive, as it is related to the Langlands functoriality conjectures. Less popular and much simpler is the case of unipotent $G=U$, which this paper is concerned with. Still, the local analogues of this problem have been studied in \cite{B1} and \cite{B2} for archimedean fields, and in \cite{Mat20} for $p$-adic fields. A common feature is that 
if $v$ is a place of $k$ and $\pi_v^\infty$ is a smooth irreducible representation of $U(k_v)$, then the space $\Hom_{U^+(k_v)}(\pi_v^\infty,\C)$ (where we ask the linear functionals to be continuous in the archimedean case) is at most one dimensional and it is nonzero if and only if 
$(\pi_v^\infty)^\vee\simeq (\pi_v^\infty)^\sigma$. Here we prove that if $\Pi$ is a closed irreducible subspace of $L^2(U(k)\backslash U_\A)$, then $p^+$ is nonvanishing on the smooth vectors $\Pi^\infty$ of $\Pi$, which we reword as $\Pi$ being distinguished, if and only if $\Pi$ is conjugate self-dual: $\Pi^\vee=\Pi^\sigma$. Hence the local results stated above imply that $p^+$ is nonvanishing on $\Pi^\infty$ if and only if each of its local components $\pi_v^\infty$ is $U^+(k_v)$-distinguished for all places $v$ of $k$. In fact because of the very strong rigidity property satisfied by the automorphic representations of $U_\A$, the representation $\Pi$ is conjugate self-dual if and only if one of its local component is, hence $\Pi$ is distinguished if and only if it has one distinguished local component. The proof is summarized in the following sentence: the period integral $p^+$ composed with the Richardson intertwining operator defined in this paper from the Kirillov induced model of $\Pi^\infty$ to $L^2(U(k)\backslash U_A)$ unfolds to the product of the local invariant linear forms on the Kirillov induced model. \\

\noindent \textbf{Acknowledgements.} We thank Abderrazak Bouaziz and Pierre Torasso 
for useful discussions as well as providing us with necessary references. We thank Michel Duflo for a stimulating exchange concerning the paper \cite{Moore65} of Moore, and Dipendra Prasad for useful comments. We thank Pierre-Henri Chaudouard for suggesting to use restriction of scalars to extend Moore's multiplicity one result to number fields rather than reproducing Moore's proof in this generality. We thank the CNRS for granting us a "délégation" in 2022 from which this work benefited.

\section{Preliminary results}

In this paper $U$ denotes a (necessarily connected) unipotent 
algebraic group defined over a field $F$ of characteristic zero, with Lie algebra $\U$ equipped with the Lie bracket $[-,-]$. 

\subsection{The exponential map}

 The exponential map \[\exp:\U\rightarrow U\] is an isomorphism of algebraic $F$-varieties with reciprocal map $\ln$ 
(\cite[Proposition 4.1]{DG}). It restricts as a bijection from 
$\U(F)$ to $U(F)$. We call $U'(F)$ an $F$-subgroup of $U(F)$ if it is the $F$-points of a closed algebraic subgroup $U'$ of $U$ defined over $F$. The map $\exp$ 
induces a bijection between Lie sub-algebras of $\U(F)$ (resp. $\U$) and the $F$-subgroups of $U(F)$ (resp. $U$), for which ideals correspond to normal subgroups. Moreover if $U'$ is an $F$-subgroup of 
$U$ then $U'(F)\backslash U(F)\simeq (U'\backslash U)(F)$ by \cite[14.2.6]{Spr}, and this bijection becomes a group isomorphism if $U'(F)$ is normal in $U$ in which case both quotients identify to 
$\frac{\U(F)}{\U'(F)}= \frac{\U}{\U'}(F)$ via $\exp$.  

We denote by $\Z$ the center of $\U$, and by $Z$ the center $\exp(\Z)$ of $U$. The nilpotency of $\U$ implies that $\Z$ is non trivial, and from this one easily deduces by induction, reasonning on whether it contains $\Z$ or not, that any codimension $1$ Lie sub-algebra of $\U$ is automatically an ideal. 

The fundamental example of non abelian unipotent group is the Heisenberg group \[U(F)=H_3(F)=\{h(x,y,z):=\begin{pmatrix} 1 & x & z \\ & 1 & y\\ & & 1\end{pmatrix}, \ x, \ y, \ z \in F\}.\] 
We will denote by \[L(F)=\{h(0,y,z), \ y, \ z \in F\}\] its normal Lagrangian subgroup. 

\subsection{Complementary basis of a Lie subalgebra}\label{subsec compl basis}

Let $\M(F)$ be a Lie subalgerba of $\U(F)$. We follow \cite[Definition 2.1.1]{B1} for the definition of a 
complementary basis of $\M(F)$ in $\U(F)$ (see also \cite[2. Definition]{CGP} for the related and more or less equivalent notion of weak basis of $\U(F)$ through $\M(F)$). Set $d:=\dim(\U(F))-\dim(\M(F))$. Then by definition, \textit{a complementary basis of $\M(F)$ in $\U(F)$} is a family $(\u_{1},\dots,\u_{d})$ in $\U(F)^{d}$ such that $\U(F)=\M(F)\oplus \Vect_F(\u_{1},\dots,\u_{d})$ and such that 
$\M(F)\oplus \Vect_F(\u_{1},\dots,\u_{d-i})$ is a Lie subalgebra of $\U(F)$ for $i=0,\dots,d$. Existence of such basis is easy to prove by induction (see for example \cite[Lemme 2.3.1]{B1}, the proof of which is valid over $F$). If 
$\B:=(\u_{1},\dots,\u_{d})$ is such a basis, then by \cite[Lemme 2.1]{B1}, the map:

\[\Phi_{\B,\M}:(\m,t_1,\dots,t_d)\in \M(F)\times F^d\mapsto \exp(\m)\exp(t_1 \u_1)\dots \exp(t_d \u_d) \in U(F)\] 
is bijective and polynomial, and its inverse map as well. In particular the map 
\[I_{\B,\M}:(t_1,\dots,t_d)\in F^d\mapsto \overline{\Phi_{\B}(0,t_1,\dots,t_d)}\in M(F)\backslash U(F)\] is a homeomorphism 
whenever $F$ is a local field of characteristic zero.

\subsection{Polarizations and restriction of scalars}\label{subsec pol and res}

For $\ell\in \U(F)^*=\Hom_F(\U(F),F)$, we say that a Lie subalgebra $\V(F)$ of $\U(F)$ is subordinate to $\ell$ if 
$[\V(F),\V(F)]\subseteq \Ker(\ell)$. We will moreover say that the pair $(\ell,\V(F))$ is polarized if $\V(F)$ is subordinate to $\ell$ and of maximal dimension for this property. We denote by $\mathcal{P}(\U(F))$ the set of polarized pairs for $\U(F)$. The Lie algebra $\U(F)$ acts on $\U(F)^*$ via the co-adjoint action. 
 We recall the following fact from \cite[Lemma 5.2]{K}, the proof of which applies for unipotent $F$-groups.

\begin{fact}\label{fact dim}
If $(\ell,\V(F))\in \mathcal{P}(\U(F))$, then the coadjoint orbit 
$U(F)\cdot \ell$ has even dimension and \[\dim(\V(F))=\dim(\U(F))-\dim(U(F)\cdot \ell)/2.\]
\end{fact}

We denote by $[\ell]$ the $U(F)$-orbit of $\ell\in \U(F)^*$. Suppose that $F/K$ is a finite extension, and denote by $\Res_{F/K}$ the Weil restriction of scalars functor. For $\ell\in \Hom_F(\U(F),F)$, we set $\ell_K:=\tr_{F/K} \circ \ell\in \Res_{F/K}\U(K)^*:=\Hom_K(\U(F),K)$. The group $U(F)=\Res_{F/K}U(K)$ acts on $\Res_{F/K}\U(K)^*$ by the co-adjoint action as well, and we again use the notation $[\ell_K]$ for the orbit of $\ell_K\in \Res_{F/K}\U(K)^*$. 

\begin{lemma}\label{lm orbits vs res scal}
The $K$-linear map $\ell\mapsto \ell_K$ is a bijection between $\U(F)^*$ and $\Res_{F/K}\U(K)^*$, hence induces an injection from  $[\ell]\mapsto [\ell_K]$ from 
$U(F)\backslash \U(F)^*$ to $\Res_{F/K}U(K)\backslash \Res_{F/K}\U(K)^*$.
\end{lemma}
\begin{proof}
The map $\ell\mapsto \ell_K$ is injective because 
$\ell_K=\ell'_K$ means that $\mathrm{Im}(\ell-\ell')$ is an $F$-vector subspace of $\Ker(\tr_{F/K})$, hence must be zero. It is bijective because both $\U(F)^*$ and $\Res_{F/K}\U(K)^*$ have the same dimension 
$\dim_K(\U(F))$ over $K$. 
\end{proof}

We deduce from the above lemma the following consequence concerning polarizations. 

\begin{lemma}\label{lm pol vs res scal}
Let $\ell$ belong to $\U(F)^*$, and $\V(F)=\Res_{F/K}\V(K)$ be a Lie sub-algebra of $\U(F)$. Then 
$(\ell,\V(F))\in \P(\U(F))$ if and only if $(\ell_K,\Res_{F/K}\V(K))\in  \P(\Res_{F/K}\U(K))$.
\end{lemma}
\begin{proof}
If $\V(F)$ is subordinate to $\ell$, then it is clearly subordinate to $\ell_K$. Conversely if it is subordinate to 
$\ell_K$, then $\ell([\V(F),\V(F)])$ is an $F$-vector subspace of $\Ker(\tr_{F/K})$, hence must be trivial. Following the proof of \cite[Lemma 5.2]{K}, the Lie algebra of the stabilizer $U(F)_\ell$ of $\ell$, respectively of the stabilizer $U(F)_{\ell_K}$ of $\ell_K$, is equal to $\{x\in \U(F),\ \ell([\U(F),x])=0\}$, respectively $\{x\in \U(F),\ \ell_K([\U(F),x])=0\}$. In particular these two spaces are equal as 
$\Ker(\tr_{F/K})$ contains no nonzero $F$-vector space. As a consequence the respective stabilizers $U(F)_\ell$ and $U(F)_{\ell_K}$ of $\ell$ and $\ell_K$, which are automatically connected, are equal because so are their Lie algebras, hence 
$U(F)\cdot \ell$ and $U(F)\cdot \ell_K$ have the same dimension over $K$. Hence the equality \[\dim_K(\V(F))=\dim_{K}(\U(F))-\dim_{K}(U(F)\cdot \ell_K)/2\] holds if and only if \[\dim_F(\V(F))=\dim_F(\U(F))-\dim_F(U(F) \cdot \ell)/2,\] and we conclude thanks to Fact \ref{fact dim} again.
\end{proof}

\subsection{Kirillov's decomposition of nilpotent Lie algebras}\label{sec kirillov dec}

\textit{In this paragraph we suppose that $\Z$ is of dimension $1$}. By Kirillov's lemma (\cite[Lemma 4.1]{K}) there is a "canonical" decomposition
\[\U(F)=F.\x\oplus F.\y\oplus F.\z \oplus \W(F)\] which means that the vectors 
$\x$, $\y$, $\z$ and $\W$ have the following properties:

\begin{enumerate}
\item $\Z(F)= F.\z$.
\item $[\x,\y]=\z$.
\item $[\y, \W(F)]=\{0\}$.
\end{enumerate}

The Lie sub-algebra \[\U_0(F):=F.\y \oplus F.\z \oplus \W(F)\] is automatically a codimension $1$ ideal of $\U(F)$ (it is the codimension one Lie subalgebra of $\U(F)$ centralizing $\y$) and we set \[U_0=\exp(\U_0).\] We also set $X(F)=\exp(F.\x)$, $Y(F)=\exp(F.\y)$ and $Z(F)=\exp(F.\z)$. Note that  $F.\x \oplus  F.\y \oplus F.\z$ is a Lie algebra isomorphic to $\H_3(F)$, hence $\exp(F.\x \oplus  F.\y \oplus F.\z)$ is a closed subgroup of $U(F)$ isomorphic to $H(F)$. 
We observe that $\y$ and $\z$ are central in $\U_0(F)$ hence they belong to $\V_0(F)$ whenever $(\ell_0,\V_0(F))\in \mathcal{P}(\U_0(F))$. 
We set \[h(x,y,z)=\exp(y.\y)\exp(x\x)\exp(z.\z)\] and use $h$ to consider $H(F)$ as a subgroup of $U(F)$ which satisfies 
\[H(F)\cap U_0(F)=L(F).\] We also 
set \[x(t)=h(t,0,0), \ y(t)=h(0,t,0) \ \mathrm{and} \ z(t)=h(0,0,t)\] for $t$ in any $F$-algebra.

\subsection{Rational involutions acting on unipotent groups}

In this paragraph $U$ is equipped with an $F$-rational involution $\sigma$, and we denote by $\sigma$ again the differential of $\sigma$ at $1$: it is an involutive homomorphism of $\U(F)$. Whenever $\mathcal{E}$ is a $\sigma$-stable subspace of $\U(F)$, it decomposes as $\mathcal{E}^+(F)\oplus \mathcal{E}^-(F)$ as the sum of the $\pm 1$ eignespaces of $\sigma$. In particular any $\v\in \U(F)$ can be written uniquely as $\v^+ + \v^-$. In this situation we have the following result as a consequence of 
\cite[Lemme 2.2.1]{B1}. Note that \cite[Lemme 2.2.1]{B1} applies over $F$ as it relies on \cite[Proposition 1.1.2]{V}, and this latter reference has no assumption on the base field.

\begin{lemma}\label{lm polarisation stable}
Take $\ell\in \U(F)^*$ such that $\ell$ is trivial on $\U(F)^+$, then there is a $\sigma$-stable Lie algebra $\V(F)$ of $\U(F)$ such that 
$(\ell,\V(F))\in \P(\U(F))$.
\end{lemma}

We also observe that it follows from \cite[Lemme 2.2.1, a)]{B1} that Kirillov's decomposition can be chosen $\sigma$-stable. 

\begin{lemma}\label{lemme dec stable}
Assume that $\Z(F)$ has dimension one. Suppose that $\sigma:U\mapsto U$ is an $F$-rational involutive group homomorphism, and denote by $\sigma:\U\mapsto \U$ its differential at 
identity. Then one can choose $\Vect_F(\x)$, $\Vect_F(\y)$, $\Vect_F(\z)$ and $\W(F)$ to be $\sigma$-stable, hence $\U_0(F)$ as well.
\end{lemma}
\begin{proof}
Clearly $\Vect_F(\z)=\Z(F)$ is $\sigma$-stable, hence $\z=\z^+$ or $\z=\z^-$. Then by \cite[Lemme 2.2.1, a)]{B1} there is a $\sigma$-stable $2$-dimensional ideal of $\U(F)$ containing $\Z(F)$, in which $\Z(F)$ has a stable complement because $\sigma$ is semi-simple. We take $\y$ a basis of it, so that $\y=\y^+$ or $\y=\y^-$. Now take $\x$ in $\U(F)$ such that $[\x,\y]=\z$. Because $\sigma$ is a Lie algebra homomorphism we can suppose that either $\x=\x^+$ or $\x=\x^-$ according to the parity of $\y$ and $\z$, and we do so. 
From the proof of \cite[Lemma 4.1]{K}, the subspace of $\U_0(F)$ of $\U(F)$ consists of vectors which commute with $\y$, so it is also $\sigma$-stable. Finally $\Vect_F(\y,\z)$ is a $\sigma$-stable subspace of $\U_0(F)$, hence it admits a $\sigma$-stable complement which we take to be $\W(F)$. 
\end{proof}

\section{Irreducible representations of $U(k_v)$}

We recall that $k$ is a number field, and for a place $v$ of $k$, we denote by $k_v$ the completion of $k$ at this place. Hence $k_v$ is either $\R$, $\C$, or an extension of $\Q_p$ for some prime number $p$. When $E$ is a finite dimensional $\R$-vector space, we denote by 
$\S(E)$ denotes the usual Schwartz space of rapidly decreasing functions from $E$ to $\C$. For uniformization of notations, when $v$ is finite and $E$ is a finite dimensional vector space over $k_v$, we denote by $\S(E)$ the space of smooth functions with compact support in $E$ and values in $\C$. More generally in the non Archimedean case, if $M$ is a complex vector space, we denote by $\S(E, M)$ the space of smooth functions from $E$ to $M$ with compact support in $E$.

\subsection{Kirillov's classification for continuous unitary representations}

We fix $F=k_v$ for some place $v$ of $k$, and $\psi:F\mapsto \C_u$ a non-trivial unitary 
character of $F$. If $\ell \in \U(F)^*$ and $\V(F)$ is a Lie sub-algebra of 
$\U(F)$ subordinate to $F$, then 
\[\Psi_{\ell,V}:v\in  V(F)\mapsto \psi(\ell(\ln(v)))\] is a unitary character of $V(F)$. Note that $V(F)$ and $U(F)$ are both unipotent, and we fix a right $U(F)$-invariant measure on $V(F)\backslash U(F)$. One can then form the 
$L^2$-induced representation 
\[\Ind_{V(F)}^{U(F), L^2}(\Psi_{\ell,V})=\{f:U(F)\rightarrow \C, \ |f| \in L^2(V(F)\backslash U(F)), \lambda(v)f=\psi(v)f\},\] where $\lambda$ stands for left translation. Note that the choice of a complementary basis $\B$ of $\V(F)$ identifies 
\[\Ind_{V(F)}^{U(F), L^2}(\Psi_{\ell,V})\simeq L^2(F^d)\] where the isometric isomorphism is given by 
\[J_{\B,\V}:f\mapsto f\circ I_{\B,\V}.\] 

Kirillov's results in \cite{K} are stated for simply connected nilpotent real Lie groups which are well-known to be algebraic over $\R$, i.e. of the form $U(\R)$ for $F=\R$ (in any case we only need to know that $U(\R)$ is nilpotent and simply connected, which is obvious). It has also been observed in \cite[p.159-160]{Moore65} that Kirillov's classification of irreducible unitary representations on Hilbert spaces holds for unipotent groups defined over for $\Q_p$. From the observations in Section \ref{subsec pol and res}, it follows that that Kirillov's classification is valid over $F$ (one could also argue that Kirillov's proof applies directly over any local field, but we want to avoid this argument). Its content (\cite[Theorem 5.2]{K}) is the following.
 
\begin{thm}\label{thm kir class}
Let $F$ be a local field. 
\begin{enumerate}
\item If $\ell\in \U(F)^*$ and $\V(F)$ is subordinate to $\ell$, then $\Ind_{V(F)}^{U(F), L^2}(\Psi_{\ell,V})$ is irreducible if and only if $(\ell,\V(F))\in \mathcal{P}(\U(F))$.
\item Any continuous irreducible unitary representation of $U(F)$ is of the form 
$\Ind_{V(F)}^{U(F), L^2}(\Psi_{\ell,V})$, for $(\ell,\V(F))\in \mathcal{P}(\U(F))$.
\item If $(\ell,\V(F))$ and $(\ell',\V'(F))\in \mathcal{P}(\U(F))$, then $\Ind_{V'(F)}^{U(F), L^2}(\Psi_{\ell',V'})\simeq \Ind_{V(F)}^{U(F), L^2}(\Psi_{\ell,V})$ if and only if $[\ell']=[\ell]$.
\end{enumerate}
\end{thm}
\begin{proof}
The the theorem to is true for $K=\R$ or $\Q_p$ for some prime number $p$, as discussed before its statement. Let $\psi:F\rightarrow \C_u$ be a non-trivial character. We first suppose that it is of the form $\psi_{0,K}:=\psi_0\circ \tr_{F/K}$ for $\psi_0:K\rightarrow \C_u$ a non-trivial character of $K$. Then write $U(F)=\Res_{F/K}U(K)$. First, if $\V(F)$ is a Lie sub-algebra of $\U(F)$ and $\ell\in \U(F)^*$, then the relation $\psi_0\circ \ell_K= \psi \circ \ell$ (where we recall that $\ell_K=\tr_{F/K}\circ 
\ell$) implies that $\Ind_{V(F)}^{U(F), L^2}(\Psi_{\ell,V})=\Ind_{V(F)}^{U(F), L^2}(\Psi_{0,\ell_K,\Res_{F/K}V})$. Hence $\Ind_{V(F)}^{U(F), L^2}(\Psi_{\ell,V})$ is irreducible if and only if $(\ell,\V(F))\in \P(U(F))$ thanks to Lemma \ref{lm pol vs res scal} and this proves the first point of the theorem over $F$ for our particular choice of character. 
Now let $\pi$ be an irreducible representation of $U(F)$ and write it $\Ind_{V_0(K)}^{U(F), L^2}(\Psi_{0,\ell_0,V})$ where 
$\ell_0\in \Res_{F/K}\V(K)^*$. One can write $\ell_0=\ell_K$ thanks to Lemma \ref{lm orbits vs res scal}. No take $\V(F)$ such that $(\ell,\V(F))\in \P(U(F))$, so that $(\ell_0, \Res_{F/K}\V(K)^*)\in \P(\Res_{F/K}U(K))$ thanks to Lemma \ref{lm pol vs res scal}. Then by the third point of the Theorem for $K$, one can write 
$\pi=\Ind_{V(F)}^{U(F), L^2}(\Psi_{0,\ell_0,\Res_{F/K}V})$. It follows that $\pi=\Ind_{V(F)}^{U(F), L^2}(\Psi_{\ell,V})$, which proves the second point of the theorem over $F$, and the third point follows from Lemma \ref{lm orbits vs res scal}. Finally if $\psi:F\rightarrow \C_u$ is non-trivial, it can be written $\psi=\psi'(a \ \bullet)$ for $a\in F^\times$ and $\psi'$ factorizing through $\tr_{F/K}$, and 
$\Ind_{V(F)}^{U(F), L^2}(\Psi_{\ell,V})=\Ind_{V(F)}^{U(F), L^2}(\Psi'_{a^{-1}\ell,V})$, hence Kirillov's classification holds for such $\psi$. 
\end{proof}

The above theorem makes the following notation legit:

\[\pi_2([\ell],\psi)\simeq \Ind_{V(F)}^{U(F), L^2}(\Psi_{\ell,V}),\] where $V$ is any $F$-subgroup of $U$ such that $(\ell,\V(F))\in \mathcal{P}(U(F))$.

\subsection{Smooth irreducible representations}\label{subsec sm rep}

When $F$ is a local field, smooth vectors of (topologically) irreducible infinite dimensional unitary representations of $U(F)$ live in some Schwartz space. For example the smooth vectors of the irreducible infinite dimensional representations of the real Heisenberg group in 3 variables are the Schwartz space of $\R$. Similarly in the non Archimedean case, the smooth vectors are given by the compact induced representation inside the $L^2$-induced representation (the full and compact induced representation coincide). Let us state the details. First for $\pi_2([\ell],\psi)$ as above, we introduce the notation \[\pi([\ell],\psi):=\pi_2([\ell],\psi)^\infty\] for the $U(F)$-invariant subspace of smooth vectors. 

\subsubsection{Smooth Archimedean representations}\label{subsec SAR}

We recall that we have fixed $\psi:F\rightarrow \C_u$ a non-trivial character. For $F$ archimedean we have the following result thanks to \cite[3.1 Corollary]{CGP} and \cite[Theorem 3.4]{Poulsen}.

\begin{thm}\label{thm arch sm}
Let $(\ell,\V(F))$ be a polarized pair for $\U(F)$. We recall that any choice of complementary basis $\B$ of $\V(F)$ in $\U(F)$ gives an isometric isomorphism $J_{\B,\V}:\pi_2([\ell],\psi)\simeq L^2(F^d)$. In fact by restriction, 
the map $J_{\B,\V}$ identifies $\pi([\ell],\psi)$ with the Schwartz space $\S(F^d)$. Moreover $\pi([\ell],\psi)$ determines $[\ell]$. 
\end{thm}

\begin{rem}
By \cite[Proposition 2;6]{Bruhat}, the representation $\pi([\ell],\psi)$ is topologically irreducible. In fact by \cite[5.1 Théorème]{DC} any smooth tempered irreducible representation of $U(F)$ is of this form. 
\end{rem}

\subsubsection{Smooth non Archimedean representations}\label{subsec NArep}

Here $F$ is a finite extension of $\Q_p$, and we consider $(\ell,\V(F))\in \P(\U(F))$. In particular 
$\pi_2([\ell],\psi)\simeq \Ind_{U(F)}^{V(F),L^2}(\Psi_{\ell,V})$. Hence $\pi([\ell],\psi)\simeq \Ind_{U(F)}^{V(F)}(\Psi_{\ell,V})$ where $\Ind_{U(F)}^{V(F)}$ denotes the usual smooth induction for smooth representations of $p$-adic groups. We write 
$\ind_{U(F)}^{V(F)}$ for compact induction. The following result is proved in \cite[Theorem 3.6]{Mat20} (we refer to \cite{Mat20} rather than the published version \cite{MatringeBIMS} as the latter states Kirillov's classification of smooth irreducible representations of $U(F)$ in a wrong manner, and contains a few other small inaccuracies, see Appendix \ref{app err} of the present paper for an erratum). We briefly recall the arguments.

\begin{lemma}
With the above notations, the representation \[\Ind_{V(F)}^{U(F)}(\Psi_{\ell,V})\] is irreducible, hence equal 
to \[\ind_{V(F)}^{U(F)}(\Psi_{\ell,V}).\] 
\end{lemma}
\begin{proof}
The statement of \cite[Theorem 3.6]{Mat20} refers directly to the compactly induced representation, i.e. states that $\ind_{V(F)}^{U(F)}(\Psi_{\ell,V})$ is irreducible. However it is mentioned at the end of 
\cite[Section 2]{Mat20} that whenever $\ind_{V(F)}^{U(F)}(\Psi_{\ell,V})$ is admissible, then it is equal to $\Ind_{V(F)}^{U(F)}(\Psi_{\ell,V})$, and we refer to \cite[Chapitre I, 5.6]{Vbook} for a proof of this fact. On the other hand the admissibility of any smooth irreducible representation of $U(F)$, and in particular of $\ind_{V(F)}^{U(F)}(\Psi_{\ell,V})$, is proved in \cite[Corollary 3.3]{Mat20} (and in \cite{VD} in a slightly different manner)
\end{proof}

\begin{rem}\label{rem mat}
By \cite[Theorem 3.6]{Mat20}, any smooth irreducible representation of $U(F)$ is of the form $\pi([\ell],\psi)$ for a unique $[\ell]$.
\end{rem}

From this it follows that given a basis $\B$ complementary to $\V(F)$ in $\U(F)$, the map 
$J_{\B,\V}$ induces an isomorphism $J_{\B,\V}:\ind_{U(F)}^{V(F)}(\Psi_{\ell,V})\simeq \S(F^d)$, where we recall that $d$ is the codimension of $\V(F)$ inside $\U(F)$. 

\subsubsection{Local tensor products}

Here $k$ a number field, and we set $F=k$. Let $S$ be a set of places of $k$, and fix $\psi_v:k_v\rightarrow \C_u$ a non-trivial character. We set $k_S=\prod_{v\in S} k_v$ and set \[\psi_S:=\otimes_{v\in S}\psi_v:k_S\rightarrow \C_u.\]
For each $v\in S$ we consider $(\ell_v,\V_v(k_v))\in \P(\U(k_v))$. We set \[V_{k_S}:=\prod_{v\in S} V_v(k_v),\] 
\[\ell_S:=\prod_{v\in S} \ell_v: \prod_{v\in S} \V_v(k_v) \rightarrow k_S,\] and 
\[\Psi_{\ell_S, V_S}:= \psi_S \circ \ell_S \circ \ln: V_{k_S}\rightarrow \C_u.\] 
We also denote by $d_v$ the codimension of 
$\V_v(k_v)$ inside $\U(k_v)$. We denote by 
\[\pi_2([\ell_S],\psi_S):=\widehat{\otimes}^2 \pi_2([\ell_v],\psi_v)\] the irreducible unitary representation of the locally compact group 
$U(k_S):=\prod_{v\in S} U(k_v)$ obtained as the completed tensor product of representations $\pi_2([\ell_v],\psi_v)$ (see \cite[Section 6]{Moore65}). Because $\widehat{\otimes}^2_{v\in S}L^2(k_v^{d_v})\simeq L^2(\prod_{v\in S} k_v^{d_v})$, it follows that 
\[\pi_2([\ell_S],\psi_S)\simeq \Ind_{V_{k_S}}^{U(k_S),L^2}(\Psi_{V_S,\ell_S}).\] We denote by 
\[\pi([\ell_S],\psi_S):=\pi_2([\ell_S],\psi_S)^\infty\] the subspace of $\pi_2([\ell_S],\psi_S)$ consisting of smooth vectors.\\

\noindent We denote by $\infty$ the set of infinite places of $k$, and suppose for a moment that $S\subseteq \infty$. It follows from 
Section \ref{subsec NArep} and \cite[(3)]{CGP} that the choice of complementary basis $\B_v$ of $\V_v(k_v)$ provides the identification 
\[J_{\B_S,\V_S}:\pi([\ell_S],\psi_S)\simeq \S(k_S).\] 
Denoting by $\widehat{\otimes}$ the projective tensor product (denoted $\widehat{\otimes}_\pi$ in \cite[Chapter 43]{Tr} and which is shown in \cite[Chapter 50]{Tr} to coincide with the tensor product denoted $\widehat{\otimes}_\epsilon$ in \cite[Chapter 43]{Tr} on nuclear Fréchet spaces). It then follows from \cite[Corollary to Theorem 51.5 and Theorem 51.6]{Tr} that 
\[\pi([\ell_S],\psi_S)\simeq \widehat{\otimes}_{v\in S} \pi([\ell_v],\psi_v).\]

If $S$ is not contained in $\infty$, then we denote by $S_\infty$ its intersection with $\infty$ and by 
$S_\mathrm{fin}$ the set of finite places it contains. In this case 
\[\pi([\ell_S],\psi_S)\simeq [\widehat{\otimes}_{v\in S_\infty} \pi([\ell_v],\psi_v)]\otimes_{v\in S_{\mathrm{fin}}} \pi([\ell_v],\psi_v),\] and we denote again by $\widehat{\otimes}_{v\in S}  \pi([\ell_v],\psi_v)$ this tensor product. 

\subsection{Local distinction}\label{sec local dist}

Here we recall results from \cite{B1}, \cite{B2} and \cite{Mat20} about $U^+(k_v)$-distinguished representations of $U(k_v)$. In the non archimedean case, we also express the local invariant linear forms as local periods for square-integrable representations. We thus fix $F$ to be a local field, and denote by $\sigma$ an $F$-rational involution of $U(F)$.

\begin{df}
Let $\pi$ be a smooth irreducible representation of $U(F)$, we say that $\pi$ is $U^+(F)$-distinguished if the space 
$\Hom_{U^+(F)}(\pi,\C)$ is not reduced to zero, where we request the elements of $\Hom_{U^+(F)}(\pi,\C)$ to be continuous with respect to the Fréchet topology on $\pi$ when $F$ is archimedean. 
\end{df}

We have the following result. 

\begin{thm}[{\cite[Proposition 2.3]{B1}, \cite[Corollaire 4.4.2]{B2}, \cite[Proposition 5.1]{Mat20}}]\label{thm dist et mult 1}
A smooth representation $\pi$ of $U(F)$ is $U^+(F)$-distinguished if and only if $\pi^\vee=\pi^\sigma$, where $\pi^\vee$ stands for the contragredient of $\pi$ and $\pi^\sigma$ for $\pi\circ \sigma$. Moreover if $\pi$ is distinguished, then 
$\dim(\Hom_{U^+(F)}(\pi,\C))=1$. 
\end{thm}

We will need in a crucial way to express in an explicit manner the local invariant linear form on the induced Kirillov model of $\pi$ for our global application. The following is a consequence of \cite[Proposition 2.3]{B1} and \cite[Proof of Theorem 5.3]{Mat20}.

\begin{prop}\label{prop explicit local linear forms}
Suppose that $\ell \in \U(F)^*$ vanishes on $\U(F)^+$ and that $(\ell,\V(F))\in \P(\U(F))$ with $\V(F)$ a $\sigma$-stable Lie subalgebra of $\U(F)$. Set $\pi=\Ind_{V(F)}^{U(F),L^2}(\Psi_{\ell,V})^\infty$. Then the space $\Hom_{U^+(F)}(\pi,\C)$ 
is spanned by the following linear form, defined by absolutely convergent integrals:
\[\lambda_{\ell,V}:=f\mapsto \int_{V^+(F) \backslash U^+(F)} f(u) du.\]
\end{prop}

\begin{rem}
In fact \cite[Proposition 2.3]{B1}, \cite[Corollaire 4.4.2]{B2} and \cite[Theorem 5.3]{Mat20} tell us that any distinguished representation of $U(F)$ can be written as an induced representation of the form given in Proposition \ref{prop explicit local linear forms}. A compact statement is that the map \[(\ell)\in U^+(F)\backslash (\frac{\U(F)}{\U^+(F)})^*\mapsto 
\pi([\ell],\psi)\] is injective with image the set of isomorphism classes of irreducible $U(F)^+$-distinguished representations of $U(F)$. 
\end{rem}

Finally, though we will not need it, it is worth noticing that when $F$ is non Archimedean, the local invariant linear forms on \textit{square-integrable} distinguished representations can also be expressed as local periods. First we observe that thanks to \cite[5, Theorem]{VD0}, the smooth irreducible \textit{square-integrable} representations of $U(F)$ are $Z$-compact in the terminilogy of \cite{MT22}. If $\pi$ is a smooth irreducible representation of 
$U(F)$, this means that if $v\in \pi$ and $v^\vee \in \pi^\vee$, then the coefficient map $c_{v^{\vee},v}:u\mapsto v^\vee(\pi(u)v)$ is compactly supported on $U(F)$ mod $Z(F)$. Then the following is a special case of 
\cite[Corollary 3.4]{MT22}.

\begin{prop}
Suppose that $F$ is non Archimedean. Let $\pi$ be a smooth irreducible \textit{square-integrable} distinguished representation of $U(F)$, and fix any vector $v^\vee\neq 0 \in \pi^\vee$, then the up to scalar unique non-zero element of $\Hom_{U^+(F)}(\pi,\C)$ is given by 
\[p_{v^\vee}: v\mapsto \int_{Z^+(F)\backslash U^+(F)} c_{v^\vee,v}(u)du.\]
\end{prop}

\section{Automorphic representations}

Here we extend Moore's results from $k=\Q$ to a general number field $k$, mainly deducing them from Moore's paper, using restriction of scalars, or sometimes giving slightly simplified proofs. The proofs in \cite{Moore65} however hold directly for $k$. We also discuss the consequences 
of these results for smooth vectors, as well as intertwining operators from global induced representations to the "automorphic space".
Here $k$ is a number field and $O$ its ring of integers, and we set $F=k$, so that $U$ is a $k$-group. We fix $\psi:\frac{\A}{k}\rightarrow \C_u$ to be a non-trivial Hecke character, and denote by $\psi_v$ its local component at the place $v$ of $k$. We denote by $O_v$ the ring of integers of $k_v$ when $v$ is finite, and by $P_v$ its maximal ideal. 

\subsection{Restricted tensor products}

First we want to defined $U_{\A}$ as a restricted direct product, hence we have to fix a compact open subgroup $K_v$ of $U_v:=U(k_v)$ for almost all (meaning all except a finite number) finite 
places $v$ of $k$. One way of doing this, as observed in \cite[Section 5]{Moore65}, to fix a $k$ basis of $\B=(\u_1,\dots,\u_n)$ of $\U(k)$, and to put $K_v:=\exp(\oplus_{i=1}^n O_v \u_i)$. This defines by the Campbell-Hausdorff formula, which is a finite sum with number of terms bounded by a function of $n$, a subgroup of of $U(k_v)$ for almost all places $v$, as soon as the denominators denominators involved in the formula in question are invertible in $O_v$, as well as the coefficients of the $k$-structure of $(\U(k),[\ , \ ])$ with respect to $\B$. 
It is also observed in \cite[Section 5]{Moore65} that changing $\B$ for another $\B'$ will define the same groups $K_v$ for almost all $v$ as the change of variable from $\B$ to $\B'$ is rational, so that $\oplus_{i=1}^n O_v \u_i'=\oplus_{i=1}^n O_v \u_i$ for almost all $v$.  Hence the restricted tensor product $U_\A$ of the groups $U_v$ with respect to the groups $K_v$ is well defined and independent of the choice of $\B$.

Now we fix $(\ell,\V(k))\in \P(\U(k))$. For any place $v$ of $k$, tensoring by $k_v$ produces $(\ell_v,\V_v)\in \P(\U_v)$. It is proved in \cite[Theorem 9]{Moore65}, the proof of which holds for $k$ instead of $\Q$, that for almost all $v$, the representation $\pi_2(\l_v,\psi_v)$ is "spherical". We recall it (with slight variations), instead of deducing it by restriction of scalars to $\Q$. 

\begin{prop}{\cite[Theorem 9]{Moore65}}
For almost all places $v$ of $k$, the vector space $\pi_2([\ell_v],\psi_v)^{K_v}=\pi([\ell_v],\psi_v)^{K_v}$ is a line. 
\end{prop}
\begin{proof} 
We set $\pi_v:=\pi([\ell_v],\psi_v)$ for any place $v$. If $\Ker(\ell)\cap \Z(k)\neq \{0\}$, we can choose $\B$ in the discussion above such that $\u_1\in \Ker(\ell)\cap \Z(k)\subseteq \V(k)$, and complete it to a basis of $\U(k)$. Note that because 
$u_1$ is central, we have $\exp(t_1\u_1+\sum_{i=2}^n t_i \u_i)=\exp(t_1 \u_1)\exp(\sum_{i=2}^n t_i \u_i)$ for all $t_i\in k_v$. Setting 
$Z_{1,v}:=\exp(k_v \u_1)$, a simple computation shows that the basis $\overline{\B}=(\overline{\u_2},\dots,\overline{\u_n})$ of 
$\U_v/\Z_{1,v}$ satisfies $\exp(\oplus_{i=2}^n k_v\overline{\u_i})=U_v/Z_{1,v}$ and $K'_v:=\exp(\oplus_{i=2}^n O_v\overline{\u_i})=K_v/K_v\cap Z_{1,v}$, as soon as $\exp(\oplus_{i=1}^n O_v\overline{\u_i})$ is a group $K_v$ (i.e. for almost all $v$). In this situation $\pi_v$ is in fact a representation of $U_v/Z_{1,v}$ (because $\ell_v$ is trivial on $k_v.\u_1$) which satisfies 
$\pi_v^{K_v}= \pi_v^{K'_v}$ for lamost all $v$, and we conclude by induction. \\
If $\Ker(\ell)\cap \Z(k)=\{0\}$, then necessarily $\Z(k)$ is one dimensional. In this situation, by \cite[Lemma 5.1]{K}, we can suppose that 
$\V(k)\subseteq \U_0(k)$ and $\ell(\y)=0$. We then take $(\u_2,\dots,\u_{n-d+1})$ to be a basis of $\V(k)$ with $\u_2=\y$ and $\u_3=\z$, complete it to a basis 
$(\u_2,\dots,\u_n)$ of $\U_0(k)$, and set $\u_1=\x$. Setting $\pi_{0,v}=\ind_{\V_v}^{\U_{0,v}}(\Psi_{v,\ell_v,V_v})$ for $v$ a finite place, then $\pi_v=\ind_{\U_{0,v}}^{\V_v}(\pi_{0,v})$, and we can identify the space of $\pi_v$ with $\S(F,\pi_{0,v})$. 
Suppose that $\psi_v$ has conductor zero, which is the case for almost all places $v$. Then if $f\in \S(F,\pi_{0,v})^{K_v}$, it is in particular $O_v$-invariant on $F$, and by \cite[Lemma 3.2 and Equation (2)]{Mat20} we also see that the support of $f$ in $F$ must be contained in the orthogonal of $O_v$ with respect to $O_v$, i.e. in $O_v$. Moreover $f(1)\in \pi_{0,v}^{K_v'}$ where this time 
$K_v':=\exp(\oplus_{i=2}^n O_v \u_i)=K_v\cap U_0$. So we have a bijection $f\mapsto f(1)$ from 
${\pi_v}^{K_v}$ to $ \pi_{0,v}^{K_v'}$, and we conclude by induction again. 
\end{proof}

In particular, for almost all finite place $v$, we can fix $f_0\in \pi_2([\ell_v],\psi_v)^{K_v}$ uniquely determined by $f_0(1)=1$, and this allows one to consider the restricted tensor product 
\[\Pi_2([\ell],\psi)=\widehat{\otimes'}_v^2 \pi_2([\ell_v],\psi_v).\] The notation makes sense as if $\ell'\in [\ell]$ then $\ell'_v\in [\ell_v]$ for all $v$. It is a topologically irreducible unitary representation of $U_{\A}$. We observe that if one fixes $\V(k)$ subordinate to $\ell$ and of maximal dimension, then 
\[\Pi_2([\ell],\psi)\simeq \Ind_{V_\A}^{U_\A,L^2}(\Psi_{\ell,V}),\] where 
\[\Psi_{\ell,V}:v\in V_\A\mapsto \psi(\ell(\ln(v))).\] We now note that 
\cite[Theorem 10]{Moore65} applies with the exact same proof, relying on \cite[Lemma 7.1]{Moore65}. However we deduce it from 
\cite[Theorem 10]{Moore65} restricting sacalars to $\Q$. 

\begin{thm}{\cite[Theorem 10]{Moore65}}\label{thm unique global Kir param}
For $\ell$ and $\ell'$ in $\U(k)^*$, one has $\Pi_2([\ell],\psi)\simeq \Pi_2([\ell'],\psi)$ if and only if $[\ell]=[\ell']$. 
\end{thm}
\begin{proof}
Clearly if $[\ell]=[\ell']$ then $[\ell_v]=[\ell'_v]$ for all $v$ and $\Pi_2([\ell],\psi)\simeq \Pi_2([\ell'],\psi)$ thanks to 
Theorem \ref{thm kir class}. Conversely if $\Pi_2([\ell],\psi)\simeq \Pi_2([\ell'],\psi)$. Using restriction of scalars from $k$ to $\Q$ as we did from $F$ to $K$ in the proof of Theorem \ref{thm kir class}, we can write $\Pi_2([\ell],\psi)=\Pi_2([(a\ell)_{\Q}],\psi_0)$ for $a\in k^\times$ and $\psi_0:\frac{\A_{\Q}}{\Q}\rightarrow \C_u$ such that $\psi=\psi_{0,k}(a \ \bullet)$. The result follows from \cite[Theorem 10]{Moore65}. 
\end{proof}

\begin{rem}\label{rem Poulsen}
According to Theorem \ref{thm arch sm} and Remark \ref{rem mat}, the condition $\Pi_2([\ell],\psi)\simeq \Pi_2([\ell'],\psi)$ holds if and only if 
$\Pi_2([\ell],\psi)^\infty\simeq \Pi_2([\ell'],\psi)^\infty$. We will use this fact without necessarily mentioning it when used. 
\end{rem}

\subsection{The decomposition of $L^2(U(k)\backslash U_{\A})$}

 According to \cite[Theorem 11]{Moore65}, when $k=\Q$ the space $L^2(U(k)\backslash U_{\A})$ decomposes in a multiplicity free way as the Hilbert direct sum of the irreducible representations introduced above. The lengthy and quite involved proof of this theorem is by induction and uses Mackey theory (but for induced representations of unitary representations from normal subgroups with infinite index, so that there are canonical direct integral decompositions showing up, and which must be compared), in the spirit of inductions that we have been doing before in the paper, and the fact that the base field $k$ is $k=\Q$ plays no role in Moore's proof. Moore also fixes a specific Hecke character of $\A_{\Q}$, which we did not do here, and which does not play any specific role in his proof. Here again we prefer to use his result and restriction of scalars to $\Q$ to generalize the statement of \cite[Theorem 11]{Moore65} to any number field. 

\begin{thm}{\cite[Theorem 11]{Moore65}}\label{thm dec aut}
\[L^2(U(k)\backslash U_{\A})=\underset{[\ell]\in U(k)\backslash \U(k)^*}{\bigoplus} \Pi_2([\ell],\psi).\] 
\end{thm}
\begin{proof}
It follows at once from \cite[Theorem 11]{Moore65}, using the equality $\Pi_2([\ell],\psi)=\Pi_2([(a\ell)_{\Q}],\psi_0)$ as in the proof of 
Theorem \ref{thm unique global Kir param}.
\end{proof}

\subsection{Smooth vectors in irreducible components of $L^2(U(k)\backslash U(\A))$}

We denote by $\S(U(k)\backslash U(\A))$ the space of smooth vectors in $L^2(U(k)\backslash U(\A))$. By definition a vector 
$\varphi \in L^2(U(k)\backslash U(\A))$ if it is fixed by a compact open subgroup of $U(\A_f)$ and the map $u_\infty\in U(k_\infty)\mapsto 
\rho(u_\infty)\varphi\in L^2(U(k)\backslash U(\A))$ is smooth. For $\ell\in \U(k)^*$, we set 
\[\Pi([\ell],\psi):=\Pi_2([\ell],\psi)\cap \S(U(k)\backslash U(\A)).\] It is a smooth Fréchet space as well as dense subspace of 
$\Pi_2([\ell],\psi),$ and it it is the restricted completed tensor product of the local representations $\pi([\ell_v],\psi_v)$:
\[\Pi([\ell],\psi)\simeq \widehat{\otimes}_v' \pi ([\ell_v],\psi_v).\]

\subsection{Intertwining operators}

In this section we define intertwining operators from $\Pi([\ell],\psi)$ to $\S(U(k)\backslash U(\A))$ in the most possible naive way, following the construction of Esienstein series in the reductive case, and call them Richarson intertwining operators following \cite{CGP}. The exact analogues of these intertwiners have been considered and used in \cite{CGP} in the archimedean case, where the authors consider co-compact discrete subgroups of $U(\R)$ when $U$ is defined over $\R$ (and in fact over $\Q$). Note that in their situation Moore's multiplicity one result does not hold. Let $U_m$ be the subgroup of $\GL_m$ consisting of upper triangular unipotent matrices. For the proof of the next proposition, we identify $U$ with a $k$-sbgroup of $U_m$ for some $m \in \mathbb{N}^*$. In particular $U(O):=U\cap U_m(O)$ is a co-compact 
discrete subgroup of $U(k_\infty)$ contained in $U(k)$, according to \cite{Borel}. Note that the restricted product defining $U_\A$ can be taken with respect to the groups $U(O_v)$ for $v$ finite, and it gives the same definition as before. 

\begin{prop}\label{prop Rich}
Take $(\ell,\V(k))\in \P(\U(k))$, and realize $\Pi([\ell],\psi)$ as 
 $\widehat{\otimes}_v'\pi ([\ell_v],\psi_v)$, where $\pi ([\ell_v],\psi_v)$ is itself realized as $\Ind_{V_v}^{U_v,L^2}(\Psi_{v,\ell_v,V_v})^\infty$. Take $f\in \Pi([\ell],\psi)$. Then the series 
\[\sum_{\gamma \in V(k)\backslash U(k)} f(\gamma \ \bullet )\] converges normally on compact subsets of $U_{\A}$ (in particular on any compact $C$ such that $U_{\A}=U(k)C$), and the function 
\[\varphi_f : u \mapsto \sum_{\gamma \in V(k)\backslash U(k)} f(\gamma u )\] belongs to $\S(U(k)\backslash U_{\A})$. Morever the map 
\[\Ri:f\mapsto \varphi_f \] is an $U_{\A}$-intertwining operator.
\end{prop}
\begin{proof}
To prove this statement we can take $f$ of the form $f_\infty\otimes f_{\mathrm{fin}}$ where the notation is self-explanatory. Let 
$C$ be a compact subset of $U_\A$. We want to show that the series of the statement converges normally on $C$, and for this we can suppose $C$ to be of the form $C_\infty\times C_{\mathrm{fin}}$ up to enlarging it. Because 
$f_{\mathrm{fin}}$ only has a finite number of right translates under $C_{\mathrm{fin}}$, it is enough to show that the series \[S(\bullet):=\sum_{\gamma \in V(k)\backslash U(k)} f_{\infty}(\gamma \ \bullet )f_{\mathrm{fin}}(\gamma)\] converges normally on $C_\infty$. Now we fix  
a complementary basis $\B$ of $\V(k)$ inside $\U(k)$. A system of representatives of $V(k)\backslash U(k)$ is then given by 
$\Phi_{\B,\V}(0,k^d)$. For each finite place $v$, the map $\Phi_v=f_v\circ \Phi_{\B,\V}(0,\bullet)$ has compact support in 
$F_v^d$, contained in $O_v^d$ for almost all $v$. From this we deduce that one can index the sum defining $S$ by $I_{\B,\V}(J^d)$ for $J$ 
a fractional ideal of $O$. For each finite place $v$, there is a compact open subgroup $K_v$ of $U_v$ containing $\Phi_{\B,\V}(0,J^d)_v$, which we take to contain $U(O_v)$ for all $v$, and equal to it for almost all $v$. The subgroup 
\[\Gamma:=\{u\in U(k), \ \forall v \in S_{\mathrm{fin}},\ u_v \in K_v\}\] of $U(k)$ thus contains $U(O)=\{u\in U(k), \ \forall v \in S_{\mathrm{fin}},\ u_v \in U(O_v)\}$ with finite index, hence is 
discrete and co-compact inside $U(k_\infty)$. Moreover $I_{\B,\V}(J^d)$ identifies with a subset of $\Gamma\cap V(k)\backslash \Gamma$, and we can take the summation in the series defining $S$ on $\Gamma\cap V(k)\backslash \Gamma=\Gamma\cap V(k_\infty)\backslash \Gamma$ to prove normal convergence. The result now follows from \cite[(4)]{CGP}. 
\end{proof}

\section{Symmetric periods}

Now $U$ is again defined over a number field $k$ and we denote by $\sigma$ a $k$-rational involuton of $U$. We recall that we denote by $p^+$ the period $p^+:\varphi \mapsto \int_{U^+(k)\backslash U^+(\A)} \varphi(u^+)du^+$ on 
$\S(U(k)\backslash U(\A))$. 

\begin{df} 
If $\Pi$ is an irreducible subspace of $L^2(U(k)\backslash U_\A)$, we say that the representation 
$\Pi$ is distinguished if $p^+$ does not vanish on $\Pi^\infty$. 
\end{df}

We set $\Pi^\sigma=\Pi\circ \sigma$, and denote by $\Pi^\vee$ the contragredient representation of $\Pi$. The local results on distinction and the decomposition of $\Pi$ into a restricted tensor product imply at once the following.

\begin{prop}\label{prop dist implies cdual}
Let $\Pi$ be an irreducible subspace of $L^2(U(k)\backslash U_\A)$ such that $\Pi$ is distinguished, then we have 
$\Pi^\vee\simeq \Pi^\sigma$.
\end{prop}
\begin{proof}
The period $p^+$ induces a nonzero $U^+(k_v)$-linear form on each $\pi_v$ for all places $v$ of $k$. The result now follows from Theorems 
\ref{thm dec aut} and \ref{thm dist et mult 1} (and Remark \ref{rem Poulsen}).
\end{proof}

We will call a local or global representation satisfying the relation of Proposition \ref{prop dist implies cdual} above \textit{conjugate self-dual}. Note that $U^+(k)$ acts naturally on $(\frac{\U(k)}{\U^+(k)})^*$ and we denote by $(\ell)$ the orbit of 
$\ell\in (\frac{\U(k)}{U^+(k)})^*$ for such an action. The following is a consequence of the fixed point result proved in 
\cite[Lemma 4.2]{Mat20}.

\begin{prop}\label{prop param of cdual}
The map $(\ell)\mapsto \Pi_2([\ell],\psi)$ is a bijection from $U^+(k)\backslash (\frac{\U(k)}{\U^+(k)})^*$ to the set of conjugate self-dual irreducible subspaces of $L^2(U(k)\backslash U_{\A})$.
\end{prop}
\begin{proof}
In view of Theorem \ref{thm unique global Kir param}, this follows from \cite[Section 4]{Mat20}, which is valid over any field of characteristic zero. 
\end{proof}

Finally if $\Pi([\ell],\psi)$ is conjugate self-dual, we can suppose that $\ell$ is trivial on $\U^+(k)$ thanks to Proposition 
\ref{prop param of cdual}, and we can also suppose that $\Pi_2([\ell],\psi)\simeq \Ind_{V_\A}^{U_\A,L^2}(\Psi_{\ell,V})$ where $V$ is $\sigma$-stable thanks to Lemma \ref{lm polarisation stable}. In this induced model, one can consider the $U_\A^+$-invariant linear form 
\[\Lambda_{\ell,V}:f\mapsto \int_{V_\A^+ \backslash U_\A^+} f(u^+) du^+\] on the smooth vectors 
$\Ind_{V_\A}^{U_\A,L^2}(\Psi_{\ell,V})^\infty \simeq \Pi([\ell],\psi)$. It is well-defined and nonzero thanks to Proposition \ref{prop explicit local linear forms}, as it factorizes into the product of the local linear forms $\ell_{\ell_v,V_v}$. Moreover by local multiplicity at most one it spans the space $\Hom_{U_\A^+}(\Pi([\ell],\psi),\C)$ (where we impose the linear forms to be continuous with respect to the Fréchet topology) thanks to Theorem \ref{thm dist et mult 1}. Of course by Theorem \ref{thm dist et mult 1}, if $\Pi$ is an irreducible subspace of $L^2(U(k)\backslash U_{\A})$ such that $\Hom_{U_\A^+}(\Pi([\ell],\psi),\C)\neq \{0\}$, then it is conjugate self-dual.
To summarize:

\begin{prop}
If $\Pi$ is an irreducible subspace of $L^2(U(k)\backslash U_{\A})$, then it is conjugate self-dual if and only if $\Hom_{U_\A^+}(\Pi^\infty,\C)$ is nonzero, in which case it is one dimensional. In this case one can realize $\Pi$ as 
$\Ind_{V_\A}^{U_\A,L^2}(\Psi_{\ell,V})$ where $\ell$ is trivial on $\U^+(k)$ and $V$ is $\sigma$-stable. Moreover the space 
$\Hom_{U_\A^+}(\Pi^\infty,\C)$ is one dimensional, and spanned by $\Lambda_{\ell,V}$. 
\end{prop}

\subsection{The case of the Heisenberg group}

In this section we prove our main result in the case of the Heisenberg group $H_3$. This is not really needed but the key "unfolding" computation for the general case shows up very clearly in the case of the Heisenberg group, hence we think that it can be useful to the reader to treat this crucial special case first. 

\begin{prop}
Let $\Pi$ be an irreducible submodule of $L^2(H_3(k)\backslash H_{3,\A})$, then $\Pi$ is distinguished if and only if 
$\Pi^\vee=\Pi^\sigma$.
\end{prop}
\begin{proof}
Thanks to Proposition \ref{prop dist implies cdual}, we can suppose that $\Pi$ is conjugate self-dual, and it remains to prove that it is distinguished. Then Proposition \ref{prop param of cdual} tells us that we can suppose that $\Pi$ is of the form $\Pi([\ell],\psi)$ with $\ell_{|\H_3^+(k)}=0$. Moreover by Lemma \ref{lemme dec stable}, we can suppose that $X(k)$, $Y(k)$, and of course $Z(k)$ are $\sigma$-stable. If $\ell_{|\Z(k)=0}$, then $\Pi$ is one dimensional and the result is an easy exercise. Hence we suppose that 
$\Ker(\ell)\cap \Z(k)=\{0\}$. This assumption implies that that $\Z^+(k)=\{0\}\Longleftrightarrow \Z(k)=\Z^-(k)$. The relation 
$[\x,\y]=\z$ then implies that $\x$ and $\y$ are of different parity. Clearly the commutative Lie algebra $\mathcal{L}(k)$ is subordinate to $\ell$ and of maximal dimension for this property, and we write $\Psi$ for $\Psi_{\ell,L}$. In particular 
$\Psi$ is trivial on $L(k)L(\A)^+$. We set $\psi_Z:=t\mapsto \Psi(z(t))$. 

We want to show that $p^+\circ \Ri$ is nonzero on $\Ind_{L_\A}^{H_{\A},L^2}(\Psi)^\infty$ and we will in fact show that it is a nonzero multiple of $\Lambda_{\ell,L}$. 
Indeed, for $f\in \Ind_{L_\A}^{H_{\A},L^2}(\Psi)^\infty$: \[p^+\circ \Ri (f)= \int_{H^+(k)\backslash {H_\A}^+}\sum_{t \in k} f(x(t) h) dh=
\int_{X^+(k) L_\A^+ \backslash H_\A^+}\int_{ X^+(k)L^+(k) \backslash X^+(k)L_{\A}^+}
\sum_{t \in k} f(x(t) l h) dl dh .\]
If $X^+=X$, we obtain 
\[p^+\circ \Ri(f)= \int_{X^+(k) L_\A^+  \backslash H_\A^+}\int_{L^+(k) \backslash X^+(k)L_\A^+ }
 f(l h) dl dh= \int_{L^+(k) \backslash H_\A^+}
 f(h) dh \]
\[=\int_{L_\A^+\backslash H_\A^+}\int_{L^+(k)\backslash L_\A^+}
 f(lh) dldh=c\int_{L_\A^+  \backslash H_\A^+} f( h)  dh\] for some positive $c$. 
If not then $X=X^-$ and $Y^+ =Y$, and we obtain 
\[p^+\circ \Ri(f)=\int_{L_\A^+  \backslash H_\A^+}\int_{L^+(k) \backslash L_\A^+ }
\sum_{t \in k} f(x(t) l h) dl dh=\]
\[=\int_{L_\A^+  \backslash H_\A^+}\int_{Y(k) \backslash Y_\A }
\sum_{t \in k} f(x(t) l h) dl dh=\int_{L_\A^+ \backslash H_\A^+}\int_{a\in \frac{\A}{k}}
\sum_{t \in k} f( x(t) y(a) h) da dh \]
\[=\int_{L_\A^+ \backslash H_\A^+}\int_{a\in \frac{\A}{k}}
\sum_{t \in k} \psi(ta) f( x(t) h) da dh = \int_{L_\A^+  \backslash H_\A^+}
\sum_{t \in k} (\int_{a\in \frac{\A}{k} } \psi(ta) da ) f( x(t) h) dh\] because the sum is uniformly convergent on compact sets. This latter sum equals  
\[=c\int_{L_\A^+ \backslash H_\A^+} f( h)  dh\] for some positive $c$ because $\int_{a\in \frac{\A}{k} } \psi(ta) da=0$ except for $t=0$. 
\end{proof}

\subsection{Symmetric periods for unipotent groups}

Here $U$ is a unipotent group defined over $k$. We do the general unfolding computation to obtain our main result. 

\begin{thm}\label{thm main}
Let $\Pi$ be an irreducible submodule of $L^2(U(k)\backslash U_\A)$, then $\Pi$ is distinguished if and only if 
$\Pi^\vee=\Pi^\sigma$.
\end{thm}
\begin{proof}
We write $\Pi=\Pi([\ell],\psi)$ with $\ell_{|\U^+(k)}=0$, and take $\V(k)$ subordinate to $\ell$ of maximal dimension and $\sigma$-stable, such that $\Pi$ it is induced from $V_\A$. If $\Ker(\ell)\cap \Z(k)\neq \{0\}$ we conclude by induction. Now we suppose that $\Ker(\ell)\cap \Z(k)=\{0\}$. According to Lemma \ref{lemme dec stable}, we can write $U(k)$ as a semi-direct product $U(k)=U_0(k).X(k)$ where both $U_0(k)$ and $X(k)$ are $\sigma$-stable, and we recall that $Y(k)$ can also be chosen $\sigma$-stable, which is what we do. We then choose thanks to 
Lemma \ref{lm polarisation stable} a $\sigma$-stable $\V(k)$ subordinate to $\ell_{|\U_0(k)}$ such that
$(\ell_{|\U_0(k)},\V(k))\in \P(\U_0(k))$. In this case it is automatic that $(\ell,\V(k))\in \P(\U(k))$. We identify 
$\Pi$ with $\Ind_{V_\A}^{U_\A,L^2}(\Psi_{\ell,V})$ and write $\Psi$ for $\Psi_{\ell,V}$. We recall that we want to show that $p^+\circ \Ri$ is a positive multiple of 
$\Lambda_{\ell,V}$ on $\Pi^\infty$. Take $f\in \Pi^\infty$:

\[p^+\circ \Ri(f)=\int_{U^+(k)\backslash U^+(\A)}\sum_{\gamma \in V(k)\backslash U(k)} f(\gamma u^+)du^+\]
\[=
\int_{U_0^+(\A)X^+(k)\backslash U^+(\A)}(\int_{U_0^+(k)X^+(k)\backslash U_0^+(\A)X^+(k)}(\sum_{\gamma \in V(k)\backslash U(k)} f(\gamma u_0^+ u^+))du_0^+)du^+\]
\[=\int_{U_0^+(\A)X^+(k)\backslash U^+(\A)}(\int_{U_0^+(k)\backslash U_0^+(\A)}(\sum_{\gamma \in V(k)\backslash U(k)} f(\gamma u_0^+ u^+))du_0^+)du^+=\]
\[=\int_{U_0^+(\A)X^+(k)\backslash U^+(\A)}(\int_{U_0^+(k)\backslash U_0^+(\A)}(\sum_{\l \in U_0(k)\backslash U(k)} \sum_{\mu \in V(k)\backslash U_0(k)} f(\mu \l u_0^+ u^+))du_0^+)du^+\]
\[=\int_{U_0^+(\A)X^+(k)\backslash U^+(\A)}(\int_{U_0^+(k)\backslash U_0^+(\A)}(\sum_{\l \in X(k)} \sum_{\mu \in V(k)\backslash U_0(k)} f(\mu \l u_0^+ u^+))du_0^+)du^+\]
\[=\int_{U_0^+(\A)\backslash U^+(\A)}(\int_{U_0^+(k)\backslash U_0^+(\A)}(\sum_{\l^- \in X^-(k)} \sum_{\mu \in V(k)\backslash U_0(k)} f(\mu \l^- u_0^+ u^+))du_0^+)du^+\] because we have the direct product decomposition $X(k)=X^+(k)X^-(k)$. 

If $X^-(k)$ is trivial, then the series of equalities continues as 
\[=\int_{U_0^+(\A)\backslash U^+(\A)}(\int_{U_0^+(k)\backslash U_0^+(\A)}(\sum_{\mu \in V(k)\backslash U_0(k)} f(\mu u_0^+ u^+))du_0^+)du^+\] and then by induction 
\[=c\int_{U_0^+(\A)\backslash U^+(\A)}(\int_{V^+(\A)\backslash U_0^+(\A)} f(u_0^+ u^+)du_0^+)du^+=c\int_{V^+(\A)\backslash U^+(\A)} f( u^+) du^+\] for some $c>0$. This terminates the proof in the case $X=X^+(k)$. 

We now treat the case $X(k)=X^-(k)$, and recall that $Y=Y^+(k)$ in this case. For fixed $u^+$ we have  
\[\int_{U_0^+(k)\backslash U_0^+(\A)}(\sum_{\l^- \in X^-(k)} \sum_{\mu \in V(k)\backslash U_0(k)} f(\mu \l^- u_0^+ u^+))du_0^+\]
\[=\int_{U_0^+(k)\backslash U_0^+(\A)}(\sum_{\l^- \in X^-(k)} \sum_{\mu \in V(k)\backslash U_0(k)} (\rho(u^+)f)(\mu \l^- u_0^+))du_0^+\]
As a consequence of Proposition \ref{prop Rich} the function $\sum_{\l^- \in X^-(k)} \sum_{\mu \in V(k)\backslash U_0(k)} (\rho(u^+)f)(\mu \l^- \ \bullet )$ is defined by a series which converges uniformly on compact subsets of $U(\A)$. In particular the function 
\[F:=\sum_{\mu \in V(k)\backslash U_0(k)} (\rho(u^+)f)(\mu \ \bullet)\] is well defined on $U(\A)$ and the series 
$\sum_{\l^- \in X^-(k)} F(\l^- \ \bullet)$ converges uniformly on any compact subset of $U(\A)$. This implies that we can reverse summation and integration in the following:
\[\int_{U_0^+(k)\backslash U_0^+(\A)}\sum_{\l^- \in X^-(k)}F(\l^- u_0^+)=\sum_{\l^- \in X^-(k)}\int_{U_0^+(k)\backslash U_0^+(\A)} F(\l^- u_0^+).\]
Now fix $\l^-\in X^-(k)$. Because $Y(k)$ is central in $U_0(k)$, the set $U_0^+(k)Y^+(\A)$ is a subgroup of $U_0^+(\A)$ and we can write
\[\int_{U_0^+(k)\backslash U_0^+(\A)} F(\l^- u_0^+)d{u_0^+}=
\int_{U_0^+(k)Y^+(\A)\backslash U_0^+(\A)}\int_{Y^+(k)\backslash Y^+(\A)}  F(\l^- y^+u_1^+)d y^+d u_0^+.\] 
Now \[F(\l^- y^+u_1^+)=F((\l^-y^+(\l^-)^{-1})\l^-u_1^+) = \sum_{\mu \in V(k)\backslash U_0(k)} (\rho(u^+)f)(\mu (\l^-y^+(\l^-)^{-1})\l^-u_1^+).\] 
But $\l^-y^+(\l^-)^{-1}$ commutes with $\mu$ as it is in the center of $U_0(\A)$, and for fixed $\mu$ we have 
\[(\rho(u^+)f)(\mu (\l^-y^+(\l^-)^{-1})\l^-u_1^+)=\Psi(\l^-y^+(\l^-)^{-1})(\rho(u^+)f)(\mu \l^-u_1^+).\] 
Hence
$F(\l^-y^+ u_1^+)=\Psi(\l^-y^+(\l^-)^{-1})F(\l^-u_1^+)$ and we observe from the Heisenberg case that the character $y\mapsto 
\Psi(\l^-y^+(\l^-)^{-1})$ is trivial if and only if $\l^-=1$. In particular 
\[\int_{Y^+(k)\backslash Y^+(\A)}  F(\l^- y^+u_1^+)d y^+=0\] except when $\l^-=1$. 
Thus we deduce \[\int_{U_0^+(k)\backslash U_0^+(\A)}\sum_{\l^- \in X^-(k)}F(\l^- u_0^+) du_0^+=\int_{U_0^+(k)\backslash U_0^+(\A)}F(u_0^+) du_0^+\] 
so that \[\int_{U_0^+(k)\backslash U_0^+(\A)}(\sum_{\l^- \in X^-(k)} \sum_{\mu \in V(k)\backslash U_0(k)} f(\mu \l^- u_0^+ u^+))du_0^+
=\int_{U_0^+(k)\backslash U_0^+(\A)}\sum_{\mu \in V(k)\backslash U_0(k)} (\rho(u^+)f)(\mu u_0^+)du_0^+,\] and we conclude as in the case 
where $X^-(k)$ is trivial.
\end{proof}

\subsection{The very strong rigidity property of automorphic representations}

The following result is proved in \cite[Proof of Theorem 10]{Moore65}, we reproduce its proof.

\begin{thm}\label{thm strong rigidity}
Let $\Pi$ and $\Pi'$ be topologically irreducible subspaces of $L^2(U(k)\backslash U_{\A})$. Fix $v$ a place of $k$. Then 
$\Pi'=\Pi$ if and only if ${\pi_v'}^\infty\simeq \pi_v^\infty$.
\end{thm}
\begin{proof}
One direction is obvious. Conversely if $\pi'_v\simeq \pi_v$. Write $\Pi=\Pi([\ell],\psi)$ and 
$\Pi'=\Pi([\ell'],\psi)$, so that $\pi_v^\infty=\pi([\ell_v],\psi_v)$ and ${\pi_v'}^\infty=\pi([\ell_v'],\psi_v)$ where $\ell_v$ and $\ell_v'$ are just obtained from $\ell$ by making them $k_v$-linear via scalar extension. Then by density of smooth vectors in the Hilbert completions of $\pi_{2}([\ell_v],\psi_v)$ and 
$\pi_{2}([\ell_v'],\psi_v)$ respectively, and Theorem \ref{thm kir class}, we deduce that $[\ell_v]=[\ell_v']$, i.e. that $\ell$ and $\ell'$ are conjugate by an element of $U(k_v)$. Now set $X=\{u\in U,\ u\cdot \ell=\ell'\}$ and $U_{\ell}=\{u\in U,\ u\cdot\ell=\ell\}$. The space $X$ is non empty as it has $k_v$-points, hence it is a principal homogenous space under $U$ defined over $k$, and it follows from \cite[Exposé 6, Théorème 3]{Bourbaki} that $X$ has a point over $k$, which means that $[\ell]=[\ell']$. 
\end{proof}

It then follows from Theorem \ref{thm main} and the results on local distinction recalled in Section \ref{sec local dist} that global distinction is detected by any given local component. 

\begin{cor}
$\Pi$ be a topologically irreducible subspace of $L^2(U(k)\backslash U_{\A})$. Fix $v$ a place of $k$. Then 
$\Pi$ is distinguished if and only if $\pi_v^\infty$ is distinguished.
\end{cor}

\section{Appendix: Erratum to \cite{MatringeBIMS}}\label{app err} 

The Kirillov classification in \cite{MatringeBIMS} is misstated, namely \cite[Theorem 3.6, 3)]{MatringeBIMS} is incorrect. We correct it here, as well as the incorrect proof of \cite[Theorem 5.2]{MatringeBIMS} which follows from this misstatement. We also correct \cite[Proof of Corollary 3.3]{MatringeBIMS} which contains minor problem. The full corrected version is \cite{Mat20}.\\ 

\noindent \cite[Proof of Corollary 3.3]{MatringeBIMS} should be replaced by the following:\\

\noindent By induction on $\dim(\mathrm{U})$. If $\dim(\mathrm{U})=1$ it is clear. If not, if either $\dim(\mathrm{Z})\geq 2$ or if $c_\pi$ is trivial, 
then setting $\K=\Ker(c_\pi)$, the group $\overline{\mathrm{U}}=\mathrm{U}/\Ker(c_\pi)$ has dimension smaller than that of $\U$ and we conclude by induction because $\pi$ is a representation of $\overline{\mathrm{U}}$. If $\dim(\mathrm{Z})=1$ and $c_\pi$ is nontrivial we can write $\pi=\ind_{\mathrm{U}_0}^{\mathrm{U}}(\pi_0)$ with $\pi_0$ good, thanks to Proposition 3.1. In this case $\pi_0$ must be irreducible so by induction it is unitary and admissible, from which we already conclude that $\pi$ is unitary. Moreover take a function 
$f\in \ind_{\mathrm{U}_0}^{\mathrm{U}}(\pi_0)\simeq \mathcal{C}_c^\infty(F,V_{\pi_0})$ which is fixed by a compact open subgroup $L$ of $\U$. Then by Equation (2) there is $k\in \mathbb{Z}$ depending on $L$ such that $f$ is an $\w_F^kO_F$-invariant function on $F$, and by Equation (2) it must vanish outside the orthogonal of $\w_F^kO_F$ with respect to $\chi$. Hence $f$ is determined by its values on a finite set $A$ depending on $L$ but not on $f\in \pi^L$, and moreover its 
image is a subset of the finite dimensional space 
$V_{\pi_0}^{L'}$ where $L'=\cap_{a\in A} a^{-1}L a$. This means that $\ind_{\mathrm{U}_0}^{\mathrm{U}}(\pi_0)^L$ has finite dimension so that $\pi$ is admissible.\\

\noindent Then \cite[Lemma 4.3]{MatringeBIMS} and the discussion before it must be repalced by:\\

\noindent We make $\sigma$ act on $\N^*$ by the formula: 
 \[\sigma(\phi)=-\phi^\sigma.\] Then a very special case of \cite[Lemma 4.2]{MatringeBIMS} is:

\begin{lemma}\label{lemme existence de polarisations stables}
Take $\phi\in \N^*$, then $\sigma(\phi)$ and $\phi$ are in the same $\mathrm{U}$-orbit if and only if there is a $\sigma$-fixed linear form in the $\mathrm{U}$-orbit of $\phi$, i.e. a linear form which vanishes on $\N^\sigma$.
\end{lemma}

Then \cite[Theorem 3.6, 3)]{MatringeBIMS} should be replaced by the following statement:\\

\noindent 3) Two irreducible representations $\pi(\mathrm{U}',\mathrm{U},\psi_{\phi})$ and $\pi(\mathrm{U}'',\mathrm{U},\psi_{\phi'})$ are isomorphic if and only if 
$\phi$ and $\phi'$ are in the same $\mathrm{U}$-orbit for the co-adjoint action.\\

\noindent The last two sentences of the proof of \cite[Theorem 3.6, 3)]{MatringeBIMS} should be replaced by:\\

\noindent  By induction this means that 
$\phi_{|\N_0}$ and $\phi'_{|\N_0}$ are $\U_0$-conjugate. Then it is explained just before \cite[Lemma 5.2]{K} at the end of the proof of \cite[Theorem 5.2]{K} that this implies that $\phi$ and $\phi'$ are indeed $\mathrm{U}$-conjugate.\\

\noindent One can then introduce the following notation:

\begin{nt}
The isomorphism class of the irreducible representation $\pi(\mathrm{U}',\mathrm{U},\phi)$ only depends on 
$\phi$, we set \[\pi(\psi_\phi):=\pi(\mathrm{U}',\mathrm{U},\psi_\phi).\]
\end{nt}

Finally the proofs and statements of \cite[Theorem 5.2 and Corollary 5.3]{MatringeBIMS} must be corrected as follow, thanks to Lemma \ref{polarisation stable} hereunder:\\

\noindent Note that $(\N^*)^\sigma$ and $(\frac{\N}{\N^\sigma})^*$ are canonically isomorphic, and we identify them. It is a space acted upon by $\U^\sigma$. Before stating the main theorem, we recall \cite[Lemma 2.2.1]{B2}, the proof of which is valid over $F$ (as it relies on \cite[Proposition 1.1.2]{V} which has no assumption on the field).

\begin{lemma}\label{polarisation stable}
Take $\phi\in(\frac{\N}{\N^\sigma})^*$, then there is a $\sigma$-stable Lie sub-algebra $\N'$ of $\N$ such that $(\phi,\N')$ is polarized.
\end{lemma}

We can now prove the following result.

\begin{thm}\label{thm distinction and sel-duality}
A representation $\pi\in \Irr_{\mathrm{U}^\sigma}(\mathrm{U})$ is distinguished if and only if $\pi^\vee=\pi^\sigma$. Moreover the map $\mathrm{U}^{\sigma}.\phi\mapsto \pi(\psi_\phi)$ is a bijection from 
$\mathrm{U}^\sigma\backslash (\frac{\N}{\N^\sigma})^*$ to $\Irr_{\mathrm{U}^\sigma \dist}(\mathrm{U})$.
\end{thm}
\begin{proof}
Suppose that $\pi=\pi(\psi_{\phi})\in \Irr(\mathrm{U})$ is conjugate self-dual, then $\sigma (\phi)$ and 
$\phi'$ are in the same $\U$-orbit, which must contain a $\sigma$-fixed linear form thanks to Lemma \ref{lemme existence de polarisations stables}. So we can in fact suppose that $\phi\in (\frac{\N}{\N^\sigma})^*$. In particular by Lemma \ref{polarisation stable} we can write $\pi(\psi_\phi)=\pi(\mathrm{U}',\mathrm{U},\psi_\phi)$ for $\mathrm{U}'=\exp(\N')$ which is $\sigma$-stable. The quotient $\mathrm{U}'^\sigma\backslash \mathrm{U}^\sigma$ identifies with a closed subset of $\mathrm{U}'\backslash \mathrm{U}$ and the condition 
$\phi\in(\frac{\N}{\N^\sigma})^*$ implies that $\psi_\phi$ is trivial on $\mathrm{U}'^\sigma$. 
 Then $\pi$ is distinguished, with explicit linear nonzero $\mathrm{U}^\sigma$-invariant linear form given on $\pi$ by 
\[\lambda:f\mapsto \int_{\mathrm{U}'^\sigma\backslash \mathrm{U}^\sigma} f(u)du.\] To finish the proof it remains to prove the injectivity 
of the map $\mathrm{U}^{\sigma}.\phi\mapsto \pi(\psi_{\phi})$, which is \cite[Lemma 4.4]{MatringeBIMS}.
\end{proof}

In particular in the case of the Galois involution one gets a bijective correspondence between 
$\Irr(\mathrm{U}^\sigma)$ and $\Irr_{\mathrm{U}^\sigma \dist}(\U)$. Indeed $\mathbf{U}=\Res_{E/F}(\mathbf{U}^\sigma)$ for $E$ a quadratic extension of $F$. Writing $\d$ for an element of $E-F$ with square in $F$. One can identify the space $(\N^\sigma)^*$ to the space $(\N^*)^\sigma$ by the map 
\[\phi_\sigma\rightarrow \phi\] where 
\[\phi(N+\d N')=\phi_\sigma(N').\] This yields:

\begin{cor}\label{cor correspondence}
When $E/F$ is a Galois involution, the map $\pi(\psi_{\phi_\sigma})\rightarrow \pi(\psi_{\phi})$ is 
a bijective correspondence from $\Irr(\mathrm{U}^\sigma)$ to $\Irr_{\mathrm{U}^\sigma\dist}(\mathrm{U})$
\end{cor}

\bibliography{unipotentdistinction}
\bibliographystyle{plain-fr}
\end{document}